\documentclass[11pt]{amsart}
\usepackage[margin=1in]{geometry}
\usepackage{latexsym}

\usepackage{amsfonts}
\usepackage{amsmath}
\usepackage{amssymb}
\usepackage{amsthm}
\usepackage{enumerate}
\usepackage[hang,flushmargin]{footmisc}
\usepackage{caption}
\usepackage{mathrsfs}
\usepackage{subfig}
\usepackage{mathtools}

\usepackage[dvipsnames]{xcolor}

\usepackage{yhmath}
\usepackage{graphicx}
\usepackage{epstopdf}
\usepackage{epsfig}
\usepackage{caption}

\DeclareFontFamily{U}{mathx}{\hyphenchar\font45}
\DeclareFontShape{U}{mathx}{m}{n}{
      <5> <6> <7> <8> <9> <10>
      <10.95> <12> <14.4> <17.28> <20.74> <24.88>
      mathx10
      }{}
\DeclareSymbolFont{mathx}{U}{mathx}{m}{n}
\DeclareFontSubstitution{U}{mathx}{m}{n}
\DeclareMathAccent{\widecheck}{0}{mathx}{"71}
 
\usepackage{bm} 
\usepackage[noadjust]{cite}

\makeatletter

\newtheorem{theorem}{Theorem}[section]
\newtheorem{lemma}[theorem]{Lemma}
\newtheorem{proposition}[theorem]{Proposition}
\newtheorem{corollary}[theorem]{Corollary}

\newtheorem{conjecture}[theorem]{Conjecture}

\theoremstyle{definition}
\newtheorem{definition}[theorem]{Definition}

\newtheorem{examplex}[theorem]{Example}
\newenvironment{example}
  {\pushQED{\qed}\examplex}
  {\popQED\endexamplex}

\DeclareMathOperator{\Pro}{Pro}
\DeclareMathOperator{\TPro}{TPro}
\DeclareMathOperator{\Acyc}{Acyc}
\DeclareMathOperator{\jdt}{jdt}

\newcommand{\dfn}[1]{\textcolor{blue}{\emph{#1}}}

\title{Toric Promotion}

\author{Colin Defant}
\address{Department of Mathematics, Princeton University, Princeton, NJ 08540, USA}
\email{cdefant@princeton.edu}

\begin{document}

\begin{abstract}
We introduce \emph{toric promotion} as a cyclic analogue of Sch\"utzenberger's promotion operator. Toric promotion acts on the set of labelings of a graph $G$. We discuss connections between toric promotion and previously-studied notions such as toric posets and friends-and-strangers graphs. Our main theorem provides a surprisingly simple description of the orbit structure of toric promotion when $G$ is a forest. 
\end{abstract}

\maketitle

\vspace{-.5cm}

\section{Introduction}\label{Sec:Intro}

\emph{Promotion} is a well-studied operator defined by Sch\"utzenberger \cite{Schutzenberger1, Schutzenberger2, Schutzenberger3} that acts on the linear extensions of a finite poset. Haiman \cite{Haiman} and Malvenuto--Reutenauer \cite{Malvenuto} simplified Sch\"utzenberger's approach by expressing promotion as a composition of local \emph{toggle operators} (also called \emph{Bender--Knuth involutions}). Born out of attempts to better understand the Robinson--Schensted--Knuth correspondence, promotion has grown in prominence through its beautiful connections with other areas \cite{Ayyer, Edelman, HopkinsRubey, Huang, Petersen, Poznanovic, Rhoades} and its multifaceted generalizations \cite{Ayyer, Bernstein, DefantPromotionSorting, Dilks, Dilks2, StanleyPromotion}. Several articles in the field of dynamical algebraic combinatorics have focused on determining the orbit structure of promotion acting on the linear extensions of specific posets. Most of the families of posets for which the orbits of promotion are known to have predictable sizes are discussed in Stanley's survey \cite{StanleyPromotion}; however, Hopkins and Rubey \cite{HopkinsRubey} recently discovered a new such class of posets.  

For our purposes, it will be convenient to define promotion in the context of graphs rather than posets. All graphs in this article are assumed to be simple. Let $G=(V,E)$ be a graph with $n$ vertices. A \dfn{labeling} of $G$ is a bijection $V\to [n]$, where we write $[n]$ for the set $\{1,\ldots,n\}$. Let $\Lambda_G$ be the set of labelings of $G$. Given distinct $i,j\in[n]$, write $(i\,\,j)$ for the transposition in the symmetric group $S_n$ that swaps $i$ and $j$. We define the \dfn{toggle} operator $\tau_{i,j}\colon \Lambda_G\to\Lambda_G$ by 
\begin{equation}\label{EqToggles}
\tau_{i,j}(\sigma)=\begin{cases} (i\,\,j)\circ \sigma, & \mbox{if } \{\sigma^{-1}(i),\sigma^{-1}(j)\}\not\in E; \\ \sigma, & \mbox{if } \{\sigma^{-1}(i),\sigma^{-1}(j)\}\in E. \end{cases}
\end{equation} In other words, $\tau_{i,j}$ swaps the labels $i$ and $j$ if those labels are assigned to nonadjacent vertices and does nothing otherwise. For $i\in[n-1]$, we write $\tau_i$ instead of $\tau_{i,i+1}$. We now define \dfn{promotion} to be the operator $\Pro\colon\Lambda_G\to\Lambda_G$ given by \[\Pro=\tau_{n-1}\tau_{n-2}\cdots\tau_2\tau_1.\] 

Given a labeling $\sigma\in\Lambda_G$, we obtain an acyclic orientation of $G$ by directing each edge from the vertex with the smaller label to the vertex with the larger label. This induces a natural partial order on $V$ in which $u\leq v$ whenever there is a directed path from $u$ to $v$ in the acyclic orientation. The labeling $\sigma$ is a linear extension of this poset, and applying $\Pro$ to $\sigma$ is exactly the same as applying Sch\"utzenberger's promotion to $\sigma$. 

In this article, we introduce \emph{toric promotion} as a cyclic analogue of promotion. This continues the recent line of work aimed at finding cyclic analogues of classical combinatorial notions (see \cite{AdinCyclic} and the references therein). Our use of the word ``toric'' stems from the work of Develin, Macauley, and Reiner \cite{Develin}, who defined \emph{toric posets} as cyclic analogues of posets.
A toric poset of a graph $G$ is defined to be a chamber of a certain toric hyperplane arrangement called the \emph{toric graphical arrangement} of $G$. One of the main results of \cite{Develin} states that toric posets are in bijection with \emph{flip equivalence classes} of $G$, which are equivalence classes of acyclic orientations of $G$ under an equivalence relation generated by source-to-sink and sink-to-source moves called \emph{flips}. These flips have appeared in numerous places under several guises in the literature, most notably in the work of Pretzel \cite{Pretzel}. 

\begin{definition}\label{Def:ToricPro}
Let $G=(V,E)$ be a graph with $n$ vertices. Define \dfn{toric promotion} to be the operator $\TPro\colon\Lambda_G\to\Lambda_G$ given by \[\TPro=\tau_n\tau_{n-1}\cdots\tau_2\tau_1=\tau_n\Pro,\] where we write $\tau_n$ for the toggle operator $\tau_{n,1}$. 
\end{definition}

\begin{example}\label{Exam1}
Let $G$ be a path with $5$ vertices. We can represent a labeling of $G$ as a permutation of the numbers $1,2,3,4,5$ written in one-line notation; two numbers are adjacent in the permutation if and only if the vertices with those labels are adjacent in $G$. Beginning with the labeling $45123$, we compute \[45123\xrightarrow{\tau_1}45123\xrightarrow{\tau_2}45123\xrightarrow{\tau_3}35124\xrightarrow{\tau_4}34125,\] so $\Pro(45123)=34125$. Then $\TPro(45123)=\tau_5(34125)=34521$. 
\end{example}

Each flip equivalence class of $G$ is a union of certain \emph{double-flip equivalence classes}, which were defined by the author and Kravitz in \cite{DefantKravitzFriends} in order to describe the connected components of certain \emph{friends-and-strangers graphs}. In Section~\ref{Sec:Background}, we give a more thorough discussion of toric posets, flip equivalence classes, double-flip equivalence classes, and friends-and-strangers graphs, and we explain how they relate to toric promotion. In particular, we will see that toric promotion restricts to an operator on the set of linear extensions of any fixed double-flip equivalence class $D$, and we will prove that the orbit structure of toric promotion on that set of linear extensions only depends on the flip equivalence class containing $D$. 

Our main theorem gives a remarkably simple description of the orbit structure of toric promotion when $G$ is a forest. Even when $G$ is a path graph, this result is nontrivial and surprising. Indeed, ordinary promotion on the labelings of a path graph has a fairly wild orbit structure, as evidenced by the fact that it has order $3224590642072800=2^5\cdot 3^2\cdot 5^2\cdot 7^2\cdot 11\cdot 13\cdot 19\cdot 23\cdot 37\cdot 59\cdot 67$ when the path has just $7$ vertices. 

\begin{theorem}\label{Thm:Forest}
Let $G$ be a forest with $n\geq 2$ vertices. Let $\sigma\in\Lambda_G$ be a labeling, and let $t$ be the number of vertices in the connected component of $G$ containing $\sigma^{-1}(1)$. Then the size of the orbit of $\sigma$ under toric promotion is \[(n-1)\frac{t}{\gcd(t,n)}.\] In particular, if $G$ is a tree, then every orbit of $\TPro\colon\Lambda_G\to\Lambda_G$ has size $n-1$. 
\end{theorem}  

\begin{example}
Let $G$ be the path on $4$ vertices. The action of toric promotion on $\Lambda_G$ is depicted in the following diagram: \[\includegraphics[width=\linewidth]{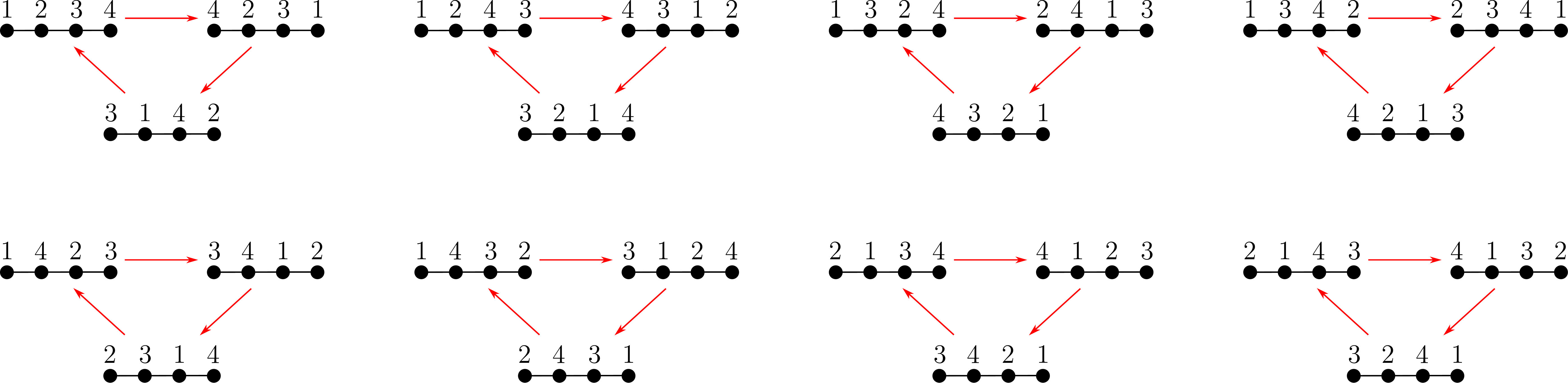}\]
\end{example}

\section{Toric Posets, Flips, and Friends-and-Strangers Graphs}\label{Sec:Background}

\subsection{Acyclic Orientations}
As before, let $G=(V,E)$ be a graph with $n$ vertices. An \dfn{acyclic orientation} of $G$ is a directed graph with no directed cycles that is obtained by assigning a direction to each edge of $G$. We write $\Acyc(G)$ for the set of acyclic orientations of $G$. A labeling $\sigma\in\Lambda_G$ induces an acyclic orientation $\alpha_\sigma$ of $G$ by orienting each edge from the vertex with the smaller label to the vertex with the larger label. A \dfn{linear extension} of an acyclic orientation $\alpha$ is a labeling $\sigma\in\Lambda_G$ such that $\alpha_\sigma=\alpha$. More generally, if $A\subseteq\Acyc(G)$, then we define a linear extension of $A$ to be a labeling $\sigma\in\Lambda_G$ such that $\alpha_\sigma\in A$; we write $\mathcal L(A)$ for the set of linear extensions of $A$. 

A \dfn{source} of a directed graph is a vertex of in-degree $0$, while a \dfn{sink} is a vertex of out-degree $0$. If $u\in V$ is a source in an acyclic orientation of $G$, then one can reverse the directions of all of the edges incident to $u$ to obtain another acyclic orientation in which $u$ is a sink. Similarly, if $v\in V$ is a sink, then reversing the orientations of the edges incident to $v$ produces another acyclic orientation in which $v$ is a source. (See Figure~\ref{Fig2}.) We call these source-to-sink and sink-to-source moves \dfn{flips}. Two acyclic orientations $\alpha,\alpha'\in\Acyc(G)$ are \dfn{flip equivalent} if $\alpha'$ can be obtained from $\alpha$ via a sequence of flips.  

\begin{figure}[ht]
\begin{center}
\includegraphics[height=1.5cm]{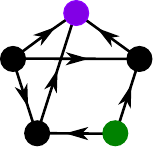} \qquad\qquad \includegraphics[height=1.5cm]{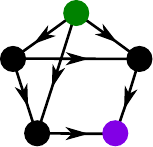}
\caption{The acyclic orientation on the right is obtained from the one on the left by performing two flips. One flip transforms the purple sink on the left into the green source on the right; the other flip transforms the green source on the left into the purple sink on the right. When performed simultaneously, these two flips constitute a double flip.}
\label{Fig2}
\end{center}  
\end{figure}

Identifying $V$ with $[n]$, we can consider the \dfn{graphical arrangement} $\mathcal A(G)$, which is the hyperplane arrangement in $\mathbb R^n$ consisting of the hyperplanes $\{(x_1,\ldots,x_n)\in\mathbb R^n:x_i=x_j\}$ for all $\{i,j\}\in E$. A \dfn{chamber} of $\mathcal A(G)$ is a connected component of $\mathbb R^n\setminus \mathcal A(G)$. It is well known that the chambers of $\mathcal A(G)$ correspond bijectively to the acyclic orientations of $G$. Develin, Macauley, and Reiner \cite{Develin} considered the \dfn{toric graphical arrangement} $\mathcal A_{\text{tor}}(G)=\pi(\mathcal A(G))\subseteq\mathbb R^n/\mathbb Z^n$, where $\pi\colon\mathbb R^n\to\mathbb R^n/\mathbb Z^n$ is the natural projection map. They defined a \dfn{toric poset} of $G$ to be a chamber (i.e., a connected component of the complement) of $\mathcal A_{\text{tor}}(G)$, and they proved that toric posets of $G$ correspond bijectively to flip equivalence classes of $\Acyc(G)$.

\subsection{Friends-and-Strangers Graphs}

Suppose $X=(V(X),E(X))$ and $Y=(V(Y),E(Y))$ are two graphs, each of which has $n$ vertices. Following \cite{DefantKravitzFriends}, we consider the \dfn{friends-and-strangers graph} of $X$ and $Y$, denoted $\mathsf{FS}(X,Y)$, which is a graph whose vertices are the bijections $\sigma\colon V(X)\to V(Y)$. Two bijections $\sigma$ and $\sigma'$ are adjacent in $\mathsf{FS}(X,Y)$ if and only if there exists an edge $\{u,v\}\in E(X)$ such that $\{\sigma(u),\sigma(v)\}\in E(Y)$, $\sigma(u)=\sigma'(v)$, $\sigma(v)=\sigma'(u)$, and $\sigma(z)=\sigma'(z)$ for all $z\in V(X)\setminus\{u,v\}$. It is straightforward to check that the map sending each bijection $V(X)\to V(Y)$ to its inverse is an isomorphism from $\mathsf{FS}(X,Y)$ to $\mathsf{FS}(Y,X)$.

\begin{example}\label{Exam2}
Suppose \[X=\begin{array}{l}\includegraphics[height=.659cm]{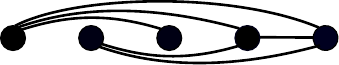}\end{array}\quad\text{and}\quad Y=\begin{array}{l}\includegraphics[height=1.5cm]{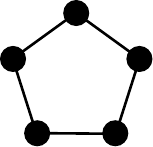}.\end{array}\]  Let the vertices of $X$ be $x_1,x_2,x_3,x_4,x_5$, listed from left to right, and let the vertices of $Y$ be $1,2,3,4,5$, listed clockwise. Each bijection $\sigma\colon V(X)\to V(Y)$ can be represented by the permutation $\sigma(x_1)\sigma(x_2)\sigma(x_3)\sigma(x_4)\sigma(x_5)$. The friends-and-strangers graph $\mathsf{FS}(X,Y)$ is shown in Figure~\ref{Fig3}. This graph has $5$ pairwise-isomorphic connected components $H_{D_1},H_{D_2},H_{D_3},H_{D_4},H_{D_5}$ (we will explain the notation shortly). Observe that $52314$ and $51324$ are adjacent in $\mathsf{FS}(X,Y)$; this corresponds to the fact that $\{x_2,x_4\}\in E(X)$ and $\{1,2\}\in E(Y)$. 
\end{example}

\begin{figure}[ht]
\begin{center}
\includegraphics[width=\linewidth]{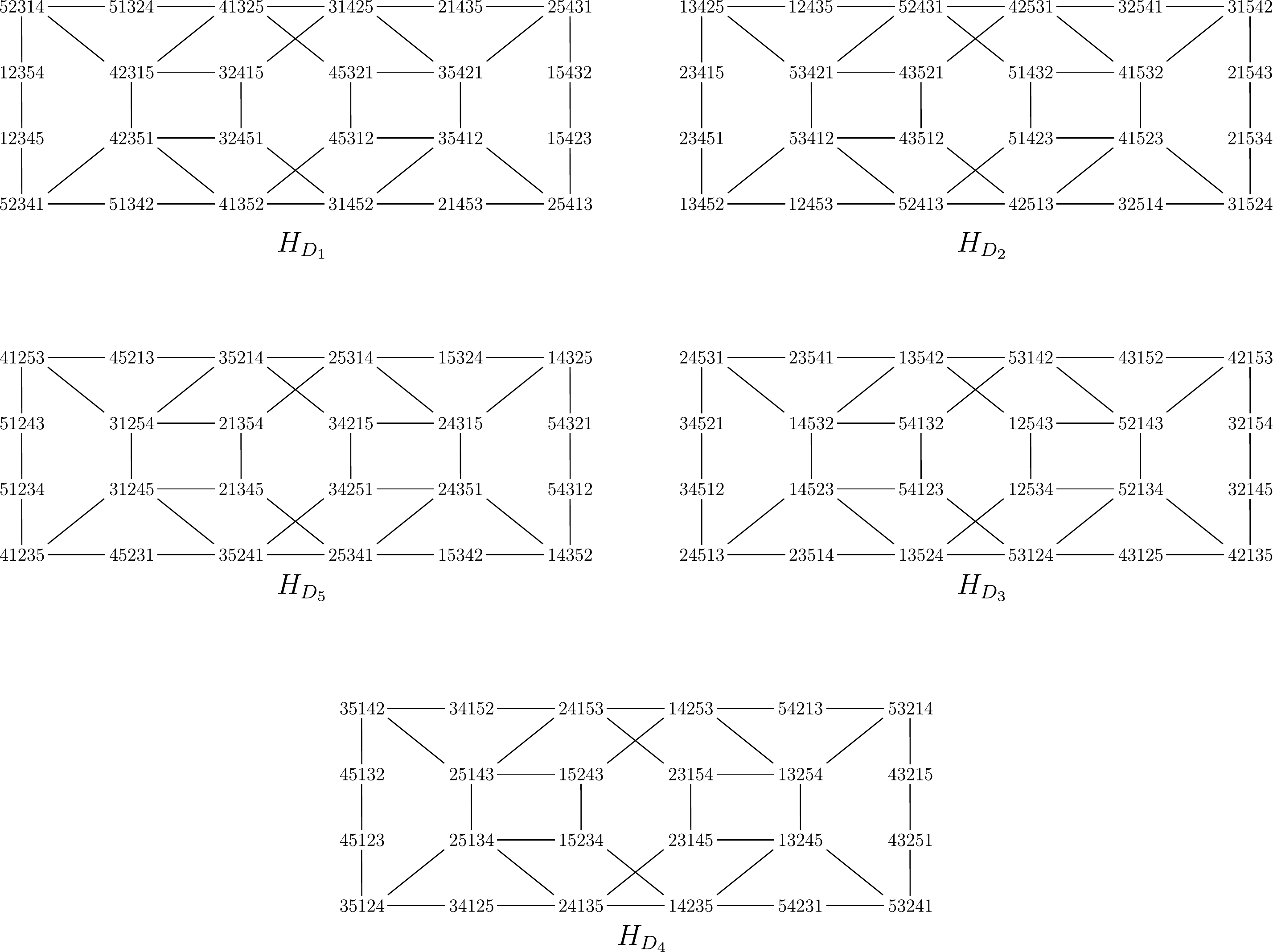} 
\caption{The friends-and-strangers graph $\mathsf{FS}(X,Y)$ from Example~\ref{Exam2}.}
\label{Fig3}
\end{center}  
\end{figure}

In this article, we are primarily interested in the case when $Y$ is the cycle graph $\mathsf{Cycle}_n$ with vertex set $V(\mathsf{Cycle}_n)=[n]$ and edge set $E(\mathsf{Cycle}_n)=\{\{i,i+1\}:i\in[n-1]\}\cup\{\{n,1\}\}$. Let $G=(V,E)$ be a graph with $n$ vertices, and let $X=\overline G$ be the complement of $G$; in other words, $X$ has the same vertex set as $G$, and two distinct vertices are adjacent in $X$ if and only if they are not adjacent in $G$. The vertices of $\mathsf{FS}(\overline G,\mathsf{Cycle}_n)$ are bijections $V\to [n]$. One readily checks that two bijections $\sigma,\sigma'\colon V\to[n]$ lie in the same connected component of $\mathsf{FS}(\overline G,\mathsf{Cycle}_n)$ if and only if there is a sequence $i_1,\ldots,i_r$ of indices in $[n]$ such that $\sigma'=(\tau_{i_r}\cdots\tau_{i_1})(\sigma)$ (with the toggle operators $\tau_{i_j}$ defined as in \eqref{EqToggles}). In particular, the toric promotion operator $\TPro\colon\Lambda_G\to\Lambda_G$ restricts to an operator on each connected component of $\mathsf{FS}(\overline G,\mathsf{Cycle}_n)$. 

The article \cite{DefantKravitzFriends} characterizes the connected components of $\mathsf{FS}(\overline G,\mathsf{Cycle}_n)$; in order to state this characterization, we need to define another equivalence relation on $\Acyc(G)$ that refines flip equivalence. Suppose $\alpha\in\Acyc(G)$ has a source $u$ and a sink $v$ such that $u\neq v$ and such that $u$ and $v$ are not adjacent in $G$ (i.e., $\{u,v\}\in E(\overline G)$). We can simultaneously flip $u$ into a sink and flip $v$ into a source; such a move is called a \dfn{double flip}. We say that two acyclic orientations of $G$ are \dfn{double-flip equivalent} if one can be obtained from the other via a sequence of double flips. For example, if $G=\overline X$, where $X$ is as in Example~\ref{Exam2}, then the $5$ double-flip equivalence classes of $\Acyc(G)$ are shown in Figure~\ref{Fig4}.

\begin{figure}[ht]
\begin{center}
\includegraphics[width=.95\linewidth]{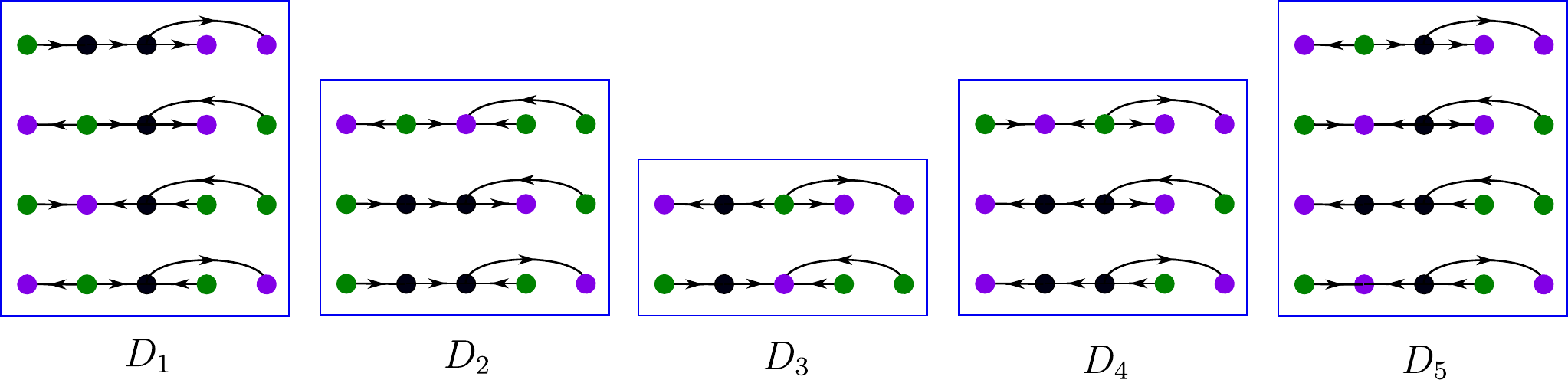} 
\caption{The graph $G=\overline X$, with $X$ as in Example~\ref{Exam2}, has $16$ acyclic orientations. Each blue box encloses one double-flip equivalence class. We have colored sinks in purple and colored sources in green.}
\label{Fig4}
\end{center}  
\end{figure}

The following theorems were proven in \cite[Section~4]{DefantKravitzFriends} for $\mathsf{FS}(\mathsf{Cycle}_n,\overline G)$, but we have stated them for $\mathsf{FS}(\overline G,\mathsf{Cycle}_n)$ instead. It is easy to transfer information back and forth between these friends-and-strangers graphs using the aforementioned isomorphism between them. 

\begin{theorem}[\cite{DefantKravitzFriends}]\label{ThmFriends1}
Let $G=(V,E)$ be a graph with $n$ vertices. Let $n_1,\ldots,n_r$ be the sizes of the vertex sets of the connected components of $G$, and let $\nu=\gcd(n_1,\ldots,n_r)$. Every flip equivalence class of $\Acyc(G)$ is a union of exactly $\nu$ double-flip equivalence classes of $\Acyc(G)$. Two bijections $V\to[n]$ lie in the same connected component of $\mathsf{FS}(\overline G,\mathsf{Cycle}_n)$ if and only if they are linear extensions of the same double-flip equivalence class of $\Acyc(G)$.  
\end{theorem}

For each double-flip equivalence class $D$ in $\Acyc(G)$, let us write $H_D$ for the connected component of $\mathsf{FS}(\overline G,\mathsf{Cycle}_n)$ whose vertex set is $\mathcal L(D)$. Let $c$ denote the long cycle $(1\,\,2\,\,3\cdots n)$ in the symmetric group $S_n$. Abusing notation, we will also write $c$ for the operator on $\Lambda_G$ given by $c(\sigma)=c\circ\sigma$. Thus, $c$ acts on a labeling by cyclically permuting the labels.  

\begin{theorem}[\cite{DefantKravitzFriends}]\label{ThmFriends2}
Let $G=(V,E)$ be a graph with $n$ vertices. Let $n_1,\ldots,n_r$ be the sizes of the vertex sets of the connected components of $G$, and let $\nu=\gcd(n_1,\ldots,n_r)$. If $C$ is a flip equivalence class of $\Acyc(G)$, then there is an ordering $D_1,\ldots,D_\nu$ of the double-flip equivalence classes contained in $C$ such that $c$ is an isomorphism from $H_{D_i}$ to $H_{D_{i+1}}$ for all $i$ (with indices taken modulo $\nu$). 
\end{theorem}

\begin{example}
Let $G=\overline X$, where $X$ is the graph from Example~\ref{Exam2}. Since $G$ is a connected graph with $5$ vertices, we have $\nu=5$. There is a unique flip equivalence class of $\Acyc(G)$. As predicted by Theorem~\ref{ThmFriends1}, this flip equivalence class is a union of $5$ double-flip equivalence classes. In Figure~\ref{Fig4}, we have labeled these double-flip equivalence classes $D_1,D_2,D_3,D_4,D_5$. The corresponding connected components $H_{D_1},H_{D_2},H_{D_3},H_{D_4},H_{D_5}$ of $\mathsf{FS}(X,\mathsf{Cycle}_5)$ are as depicted in Figure~\ref{Fig3}. One can check that $c$ is an isomorphism from $H_{D_i}$ to $H_{D_{i+1}}$ for each $i$ (modulo $5$), as predicted by Theorem~\ref{ThmFriends2}. 
\end{example}

Theorem~\ref{ThmFriends1} tells us that toric promotion on $\Lambda_G$ restricts to an operator on the set of linear extensions of a fixed double-flip equivalence class of acyclic orientations of $G$; this is analogous to the fact that promotion restricts to an operator on the set of linear extensions of a fixed acyclic orientation of $G$. The next theorem tells us that the orbit structure of toric promotion on the set of linear extensions of a double-flip equivalence class $D$ only depends on the flip equivalence class containing $D$. We first record a very simple lemma, the proof of which is immediate. 

\begin{lemma}\label{Lem:EasyRelation}
Let $G=(V,E)$ be a graph with $n$ vertices. The operators $c,\tau_1,\ldots,\tau_n\colon\Lambda_G\to\Lambda_G$ satisfy the relations $c\tau_i=\tau_{i+1} c$, where the indices are taken modulo $n$.   
\end{lemma}

\begin{theorem}\label{Thm:General}
Let $G=(V,E)$ be a graph with $n$ vertices, and let $D$ and $D'$ be double-flip equivalence classes of $\Acyc(G)$ that are contained in the same flip equivalence class of $G$. Then there exists $m\in\{0,\ldots,n-1\}$ such that $\tau_1\cdots\tau_m c^m$ is a bijection from $\mathcal L(D)$ to $\mathcal L(D')$ that commutes with $\TPro$. 
\end{theorem}

\begin{proof}
It follows from Theorem~\ref{ThmFriends2} that there exists $m\in\{0,\ldots,n-1\}$ such that $c^m$ is a bijection from $\mathcal L(D)$ to $\mathcal L(D')$. Since $\tau_1,\ldots,\tau_m$ are bijections that preserve the acyclic orientation induced by a labeling, the map $\tau_1\cdots\tau_m c^m$ is also a bijection from $\mathcal L(D)$ to $\mathcal L(D')$. Using Lemma~\ref{Lem:EasyRelation}, we find that \[c^m\TPro=c^m\tau_n\cdots\tau_1=\tau_m\cdots\tau_1\tau_n\cdots\tau_{m+1} c^m.\] Therefore, \[\tau_1\cdots \tau_m c^m\TPro=\tau_1\cdots \tau_m \tau_m\cdots\tau_1\tau_n\cdots\tau_{m+1} c^m=\tau_n\cdots\tau_{m+1} c^m=\tau_n\cdots\tau_{m+1}\tau_m\cdots\tau_1\tau_1\cdots\tau_m c^m\] \[=\TPro\tau_1\cdots \tau_m c^m. \qedhere\]
\end{proof}

\section{Toric Promotion on Forests}
The purpose of this section is to prove Theorem~\ref{Thm:Forest}, which describes the orbit structure of toric promotion when $G$ is a forest. We begin with some general lemmas that apply when $G$ is an arbitrary graph with $n$ vertices. As in the previous section, let $c\colon\Lambda_G\to\Lambda_G$ be the operator that cyclically shifts the labels in a labeling. 

\begin{lemma}\label{Lem:FactoringTPro}
Let $s$ and $k$ be integers with $s\geq 0$ and $0\leq k\leq n-2$. We have $\TPro^{(n-1)s+k}= \tau_n\tau_{n-1}\cdots\tau_{n-k+1} c^{-k}(c\Pro)^{ns+k}$.
\end{lemma}

\begin{proof}
Let $\omega_r=\tau_r\tau_{r-1}\cdots\tau_1\tau_n\tau_{n-1}\cdots\tau_{r+2}$ (with indices modulo $n$). Using Lemma~\ref{Lem:EasyRelation}, one readily checks that $\omega_r=c^{-(n-r-1)}\Pro c^{n-r-1}$. Thus, 
\begin{align}
\TPro^{n-1}&=(\tau_n\cdots\tau_1)^{n-1} \nonumber \\ &=\omega_0\omega_{1}\cdots\omega_{n-1} \nonumber \\
&=(c^{-(n-1)}\Pro c^{n-1})(c^{-(n-2)}\Pro c^{n-2})\cdots(c^{-1}\Pro c)(c^{-0}\Pro c^0) \nonumber\\ &=c^{-(n-1)}\Pro\,(c\Pro)^{n-1} \nonumber \\ &=(c\Pro)^n. \nonumber
\end{align}
This shows that $\TPro^{(n-1)s}=(c\Pro)^{ns}$. The proof now follows from the computation
\begin{align}
\TPro^k&=\tau_{n}\tau_{n-1}\cdots\tau_{n-k+1}\omega_{n-k}\omega_{n-k+1}\cdots\omega_{n-1} \nonumber \\ &=\tau_n\tau_{n-1}\cdots\tau_{n-k+1}(c^{-(k-1)}\Pro c^{k-1})(c^{-(k-2)}\Pro c^{k-2})\cdots(c^{-0}\Pro c^0) \nonumber \\ &=\tau_n\tau_{n-1}\cdots\tau_{n-k+1}c^{-k}(c\Pro)^k. \qedhere
\end{align}

\end{proof}

Lemma~\ref{Lem:FactoringTPro} motivates us to better understand $c\Pro\colon\Lambda_G\to\Lambda_G$. This operator can be described via a sliding operation analogous to jeu de taquin. Given adjacent vertices $u,v\in V$ and $\sigma\in\Lambda_G$, let $\jdt_{(u,v)}=(\sigma(u)\,\,\sigma(v))\circ \sigma$ be the labeling obtained from $\sigma$ by swapping the labels of $u$ and $v$. More generally, if $\mathcal P=(v_1,\ldots,v_r)$ is a path in $G$ (i.e., a tuple of distinct vertices such that $v_i$ is adjacent to $v_{i+1}$ for all $i\in[r-1]$), then we define \[\jdt_{\mathcal P}=\jdt_{(v_{r-1},v_r)}\cdots\jdt_{(v_2,v_3)}\jdt_{(v_1,v_2)}.\] Thus, $\jdt_{\mathcal P}(\sigma)$ is obtained from $\sigma$ by sliding the label $\sigma(v_1)$ along the path $\mathcal P$; at each step, this label swaps places with the label on the next vertex in the path. When $r=1$, we make the convention that $\jdt_{\mathcal P}(\sigma)=\sigma$ for all $\sigma$. 

The proof of the next lemma is straightforward and is essentially the same as the standard proof that Sch\"utzenberger's definition of promotion using jeu de taquin coincides with the definition of promotion as a composition of toggle operators (e.g., see the proof of Theorem~2.1 in \cite{StanleyPromotion}). See Example~\ref{ExamE} for an illustration. 

\begin{lemma}\label{Lem:jdt1}
Let $\sigma\in\Lambda_G$. Construct a sequence $j_1,\ldots,j_r$ recursively as follows. First, let $j_1=1$. 
For each $i$, define $j_{i+1}$ to be the smallest label that is larger than $j_i$ and is assigned to a vertex adjacent to $\sigma^{-1}(j_i)$, assuming such a vertex exists. This process ends when we reach a label $j_r$ that is larger than all labels assigned to vertices adjacent to $\sigma^{-1}(j_r)$. Consider the path $\mathcal P=(\sigma^{-1}(j_1),\ldots,\sigma^{-1}(j_r))$. We have $(c\Pro)(\sigma)=\jdt_{\mathcal P}(\sigma)$. 
\end{lemma}

We now specialize to the case when $G$ is a forest. 

\begin{lemma}\label{Lem:jdt2}
Let $G$ be a forest. Let $\sigma\in\Lambda_G$. For each integer $\ell\geq 0$, there exists a unique path $\mathcal P=(v_1,\ldots,v_r)$ in $G$ such that $\sigma(v_1)=1$ and such that $(c\Pro)^\ell(\sigma)=\jdt_{\mathcal P}(\sigma)$.  
\end{lemma}

\begin{proof}
When $\ell=0$, the statement is trivial (take $r=1$ so that $\jdt_{\mathcal P}(\sigma)=\sigma$). When $\ell=1$, the statement follows directly from Lemma~\ref{Lem:jdt1}. Now suppose $\ell\geq 2$. By induction, we know that there is a unique path $\mathcal P'=(v_1',\ldots,v_{r'}')$ such that $\sigma(v_1')=1$ and such that $(c\Pro)^{\ell-1}(\sigma)=\jdt_{\mathcal P'}(\sigma)$. Let $\sigma'=(c\Pro)^{\ell-1}(\sigma)$. Appealing to Lemma~\ref{Lem:jdt1} again, we see that there is a unique path $\mathcal P''=(v_1'',\ldots,v_{r''}'')$ such that $\sigma'(v_1'')=1$ and such that $(c\Pro)(\sigma')=\jdt_{\mathcal P''}(\sigma')$. Thus, $(c\Pro)^{\ell}(\sigma)=\jdt_{\mathcal P''}\jdt_{\mathcal P'}(\sigma)$. Let $\mathcal P$ be the unique path in the forest $G$ starting at $v_1'$ and ending at $v_{r''}''$. Then $(c\Pro)^{\ell}(\sigma)=\jdt_{\mathcal P''}\jdt_{\mathcal P'}(\sigma)=\jdt_{\mathcal P}(\sigma)$. Furthermore, the first vertex in $\mathcal P$ is $v_1'$, which is $\sigma^{-1}(1)$. 
\end{proof}

\begin{example}\label{ExamE} Consider the following applications of $c\Pro$ to labelings of a tree: 
\[\begin{array}{l}\includegraphics[height=4.315cm]{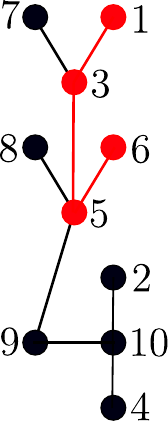}\end{array}\xrightarrow{c\Pro}\begin{array}{l}\includegraphics[height=4.315cm]{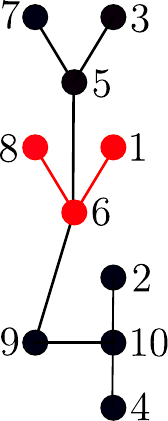}\end{array}\xrightarrow{c\Pro}\begin{array}{l}\includegraphics[height=4.315cm]{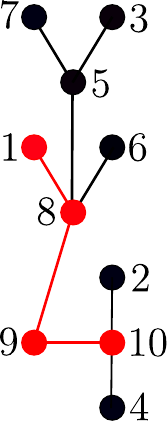}\end{array}\xrightarrow{c\Pro}\begin{array}{l}\includegraphics[height=4.315cm]{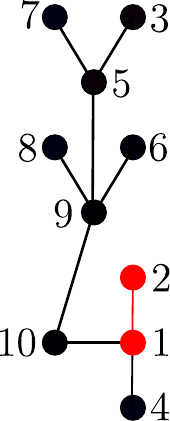}\end{array}\xrightarrow{c\Pro}\begin{array}{l}\includegraphics[height=4.315cm]{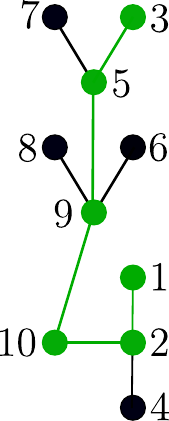}\end{array}.\] Let these labelings be $\sigma_0,\sigma_1,\sigma_2,\sigma_3,\sigma_4$, listed from left to right. Thus, $\sigma_{j}=(c\Pro)^{j}(\sigma_0)$. For each $0\leq i\leq 3$, we have drawn the image depicting $\sigma_i$ so that the path $\mathcal P$ satisfying $\sigma_{i+1}=\jdt_{\mathcal P}(\sigma_i)$ (which exists by Lemma~\ref{Lem:jdt1}) is colored red. In the image depicting $\sigma_4$, we have colored green the path $\mathcal P$ such that $\sigma_4=\jdt_{\mathcal P}(\sigma_0)$ (which exists by Lemma~\ref{Lem:jdt2}). 
\end{example}

We can now proceed to the core arguments involved in the proof of Theorem~\ref{Thm:Forest}. 

\begin{proposition}\label{Prop1}
If $G$ is a forest with $n\geq 2$ vertices, then every orbit of the operator $\TPro\colon\Lambda_G\to\Lambda_G$ has size divisible by $n-1$. 
\end{proposition}

\begin{proof}
This is trivial if $n=2$, so assume $n\geq 3$. Let $s$ and $k$ be integers with $s\geq 0$ and $1\leq k\leq n-2$ so that $(n-1)s+k$ is not divisible by $n-1$, and suppose by way of contradiction that there is a labeling $\sigma\in\Lambda_G$ with $\TPro^{(n-1)s+k}(\sigma)=\sigma$. Lemma~\ref{Lem:FactoringTPro} tells us that $\TPro^{(n-1)s+k}=\tau_n\tau_{n-1}\cdots\tau_{n-k+1} c^{-k}(c\Pro)^{ns+k}$, and Lemma~\ref{Lem:jdt2} tells us that there is a path $\mathcal P=(v_1,\ldots,v_r)$ in $G$ such that $\sigma(v_1)=1$ and such that $(c\Pro)^{ns+k}(\sigma)=\jdt_{\mathcal P}(\sigma)$. Suppose there is a vertex $x$ of $G$ that is not in $\mathcal P$. Then $\jdt_{\mathcal P}(\sigma)(x)=\sigma(x)$, so \[\sigma(x)=\TPro^{(n-1)s+k}(\sigma)(x)=(\tau_n\tau_{n-1}\cdots\tau_{n-k+1} c^{-k}(c\Pro)^{ns+k})(\sigma)(x)\] \[=(\tau_n\tau_{n-1}\cdots\tau_{n-k+1} c^{-k})(\sigma)(x).\] When we apply $c^{-k}$ to $\sigma$, the label of $x$ decreases by $k$ (modulo $n$), so applying $\tau_n\tau_{n-1}\cdots\tau_{n-k+1}$ to $c^{-k}(\sigma)$ must either increase the label of $x$ by $k$ or decrease the label of $x$ by $n-k$. The latter scenario is impossible because $n-k\geq 2$ and $\tau_n\tau_{n-1}\cdots\tau_{n-k+1}$ can only decrease a label by at most $1$. So applying this composition of $k$ toggle operators to $c^{-k}(\sigma)$ increases the label of $x$ by $k$; this can only happen if $c^{-k}(\sigma)(x)=n-k+1$. Hence, we must have $\sigma(x)=1$. This is a contradiction because $v_1=\sigma^{-1}(1)$ is in $\mathcal P$ and $x$ is not in $\mathcal P$. 

The preceding paragraph proves that every vertex in $G$ is in $\mathcal P$, so $G$ is a path graph. Furthermore, $\sigma^{-1}(1)=v_1$ must be an endpoint of $G$. In fact, this argument shows that if $\mu\in\Lambda_G$ is any labeling that is fixed by $\TPro^{(n-1)s+k}$, then $\mu^{-1}(1)$ is an endpoint of $G$. Let $y$ be the unique vertex adjacent to $v_1$; note that $y$ is not an endpoint of $G$ because $n\geq 3$. Let $m=n+1-\sigma(y)$, and let $\mu=(\tau_1\cdots\tau_m c^m)(\sigma)$. We saw in the proof of Theorem~\ref{Thm:General} that $\tau_1\cdots\tau_m c^m$ commutes with $\TPro$, so $\TPro^{(n-1)s+k}(\mu)=(\tau_1\cdots\tau_m c^m\TPro^{(n-1)s+k})(\sigma)=(\tau_1\cdots\tau_m c^m)(\sigma)=\mu$. Hence, $\mu^{-1}(1)$ must be an endpoint of $G$. Now, $c^m(\sigma)(v_1)=m+1$, and $c^m(\sigma)(y)=1$. Since $y$ is the only vertex in $G$ adjacent to $v_1$, we have $(\tau_2\cdots\tau_m c^m)(\sigma)(v_1)=2$ and $(\tau_2\cdots\tau_m c^m)(\sigma)(y)=1$. Then $\tau_1$ does nothing when we apply it to $(\tau_2\cdots\tau_m c^m)(\sigma)$, so $\mu(y)=(\tau_1\tau_2\cdots\tau_m c^m)(\sigma)(y)=(\tau_2\cdots\tau_m c^m)(\sigma)(y)=1$. This is a contradiction because it shows that $\mu^{-1}(1)=y$ is not an endpoint of $G$. 
\end{proof}

\begin{proposition}\label{Prop2}
If $G$ is a tree with $n$ vertices, then $\TPro^{n-1}(\sigma)=\sigma$ for all $\sigma\in\Lambda_G$. 
\end{proposition}

\begin{proof}
Suppose instead that there exists $\sigma\in\Lambda_G$ with $\TPro^{n-1}(\sigma)\neq\sigma$. Consider a circle with $n$ marked points labeled $1,\ldots,n$ in clockwise order; we refer to the marked point with label $i$ as \emph{position} $i$. Given a labeling $\mu\in\Lambda_G$, we imagine placing the vertices of $G$ on the marked points of the circle so that each vertex $x$ sits on position $\mu(x)$. Consider the process where we begin with the labeling $\sigma$ and apply the toggle operators in the composition $(\tau_n\cdots\tau_1)^{n-1}$ to reach the labeling $\TPro^{n-1}(\sigma)$. Applying $\tau_i$ to a labeling corresponds to swapping the vertices on positions $i$ and $i+1$ if those vertices are not adjacent in $G$ and doing nothing if they are adjacent in $G$. Thus, the entire process can be seen as a sequence of operations that move the vertices around the circle. Now imagine that we assign each vertex of $G$ a weight. At the beginning of the process (i.e., where we start with the labeling $\sigma$), every vertex has weight $0$. Whenever a vertex moves one space clockwise (respectively, counterclockwise) on the circle, we increase (respectively, decrease) its weight by $1$. Observe that the sum of the weights of all the vertices remains constant at $0$ throughout the process. Let $f(x)$ denote the final weight of the vertex $x$ at the end of the process (when we reach the labeling $\TPro^{n-1}(\sigma)$). Let \[F_+=\{x\in V:f(x)>0\},\quad F_-=\{x\in V: f(x)<0\}, \quad\text{and}\quad F_0=\{x\in V:f(x)=0\}.\]

We know by Lemma~\ref{Lem:FactoringTPro} that $\TPro^{n-1}=(c\Pro)^n$. Lemma~\ref{Lem:jdt2} tells us that there is a path $\mathcal P=(v_1,\ldots,v_r)$ in $G$ such that $\sigma(v_1)=1$ and such that $(c\Pro)^n(\sigma)=\jdt_{\mathcal P}(\sigma)$; note that $r\geq 2$ because we have assumed that $\TPro^{n-1}(\sigma)\neq \sigma$. This shows that during the process, each of the vertices $v_1,\ldots,v_r$ ends on a position that is different from the position where it started. Hence, none of the vertices $v_1,\ldots,v_r$ belong to $F_0$. Let us assume $v_r\in F_+$; the proof when $v_r\in F_-$ is virtually identical. Because $v_r$ and $v_{r-1}$ are adjacent in $G$, they can never swap positions with each other when we apply a toggle operator. Also, the final position of $v_{r-1}$ at the end of the process is $\jdt_{\mathcal P}(\sigma)(v_{r-1})=\sigma(v_{r})$. This forces $f(v_{r-1})>0$, so $v_{r-1}\in F_+$. Similarly, $v_{r-1}$ and $v_{r-2}$ can never swap positions when we apply a toggle operator, and the final position of $v_{r-2}$ at the end of the process is $\jdt_{\mathcal P}(\sigma)(v_{r-2})=\sigma(v_{r-1})$. This forces $f(v_{r-2})>0$, so $v_{r-2}\in F_+$. Continuing in this manner, we eventually find that $\{v_1,\ldots,v_r\}\subseteq F_+$. Since $r\geq 2$, the set $F_+$ is nonempty. But the sum of the weights of all the vertices at the end of the process is $0$, so $F_-$ must also be nonempty. 

Suppose $z\in F_-$. Every vertex in $\mathcal P$ belongs to $F_+$, so $z$ is not in $\mathcal P$. This means that $\TPro^{n-1}(\sigma)(z)=\jdt_{\mathcal P}(\sigma)(z)=\sigma(z)$, so the initial position of $z$ at the beginning of the process is the same as the final position of $z$ at the end of the process. Consequently, $f(z)$ is a negative multiple of $n$. Since $z$ moves (in net) a positive number of full rotations counterclockwise, it must cross paths with every vertex that (in net) stays still or moves clockwise. In other words, if $x\in V$ is any vertex satisfying $f(x)\geq 0$, then $z$ and $x$ must swap positions with each other at some time during the process, so $z$ and $x$ must not be adjacent in $G$. This argument shows that no vertex in $F_-$ can be adjacent in $G$ to any vertex in $F_0\cup F_+$. However, this contradicts the fact that $G$ is connected because the sets $F_-$ and $F_0\cup F_+$ are nonempty subsets of $V$ whose union is $V$. 
\end{proof}

Note that Propositions~\ref{Prop1} and~\ref{Prop2} imply the last statement of Theorem~\ref{Thm:Forest}; we are almost ready to complete the proof of the entire theorem. First, let us record a corollary of Proposition~\ref{Prop2}. 

\begin{corollary}\label{Cor:cPro}
If $G$ is a tree with $n$ vertices, then every orbit of $c\Pro\colon\Lambda_G\to\Lambda_G$ has size $n$. 
\end{corollary}

\begin{proof}
The proof is trivial if $n=1$, so we may assume $n\geq 2$. It is well known that a tree has exactly one flip equivalence class (see, e.g., \cite[Example~1.5]{Develin}). Therefore, it follows from Theorems~\ref{ThmFriends1} and~\ref{ThmFriends2} that $G$ has $n$ double-flip equivalence classes $D_1,\ldots,D_n$ such that $c$ is a bijection from $\mathcal L(D_i)$ to $\mathcal L(D_{i+1})$ for all $i$ (with indices taken modulo $n$). Because $\Pro$ preserves the acyclic orientation induced by a labeling (i.e., $\alpha_{\sigma}=\alpha_{\Pro(\sigma)}$), the operator $c\Pro$ also restricts to a bijection from $\mathcal L(D_i)$ to $\mathcal L(D_{i+1})$ for all $i$. It follows that every orbit of $c\Pro\colon\Lambda_G\to\Lambda_G$ has size divisible by $n$. On the other hand, we know by Lemma~\ref{Lem:FactoringTPro} that $(c\Pro)^n=\TPro^{n-1}$, and Proposition~\ref{Prop2} tells us that this is the identity map on $\Lambda_G$.  
\end{proof}

\begin{proof}[Proof of Theorem~\ref{Thm:Forest}]
Let $G$ be a forest with $n\geq 2$ vertices, and let $\sigma\in\Lambda_G$ be a labeling. Let $T=(V(T),E(T))$ be the connected component of $G$ containing $\sigma^{-1}(1)$, and let $t$ be the number of vertices in $T$. According to Proposition~\ref{Prop1}, the size of the orbit of $\sigma$ under $\TPro$ is $(n-1)s$ for some positive integer $s$; our goal is to prove that $s=t/\gcd(t,n)$. Since $\TPro^{n-1}=(c\Pro)^n$ by Lemma~\ref{Lem:FactoringTPro}, $s$ is the smallest positive integer such that $(c\Pro)^{ns}(\sigma)=\sigma$. 

Let $\ell$ be a positive integer. Lemma~\ref{Lem:jdt2} tells us that there is a path $\mathcal P=(v_1,\ldots,v_r)$ in $G$ such that $\sigma(v_1)=1$ and such that $(c\Pro)^{n\ell}(\sigma)=\jdt_{\mathcal P}(\sigma)$. The path $\mathcal P$ is contained in $T$, so, roughly speaking, applying $(c\Pro)^{n\ell}$ to $\sigma$ has the same effect as applying $(c\Pro)^{n\ell}$ to the restriction of $\sigma$ to $T$. To make this statement more precise, we need some notation. Let $B=\sigma(V(T))\subseteq[n]$, and let $\gamma$ be the unique order-preserving bijection from $B$ to $[t]$. Given a labeling $\mu\in\Lambda_G$, let $\mu\vert_T$ denote the restriction of $\mu$ to $V(T)$. Given a bijection $\beta\colon V(T)\to B$, let $\beta\vert^G$ denote the unique labeling in $\Lambda_G$ that agrees with $\beta$ on $V(T)$ and agrees with $\sigma$ on $V(G)\setminus V(T)$. Let $c_T$ be the long cycle $(1 \,\, 2\,\, 3\cdots t)$ in the symmetric group $S_t$, which we view as an operator on $\Lambda_T$. Let $\Pro_T\colon\Lambda_T\to\Lambda_T$ denote the promotion operator for the graph $T$. 
Then $\gamma\circ\sigma\vert_T\in\Lambda_T$, so
$\gamma^{-1}\circ(c_T\Pro_T)^{n\ell}(\gamma\circ\sigma\vert_T)$ is a bijection from $V(T)$ to $B$. The precise version of the rough statement given above is that $(\gamma^{-1}\circ(c_T\Pro_T)^{n\ell}(\gamma\circ\sigma\vert_T))\vert^G=(c\Pro)^{n\ell}(\sigma)$. It follows that $(c\Pro)^{n\ell}(\sigma)=\sigma$ if and only if $(c_T\Pro_T)^{n\ell}(\gamma\circ\sigma\vert_T)=\gamma\circ\sigma\vert_T$. Because $T$ is a tree with $t$ vertices, Corollary~\ref{Cor:cPro} tells us that $(c_T\Pro_T)^{n\ell}(\gamma\circ\sigma\vert_T)=\gamma\circ\sigma\vert_T$ if and only if $n\ell$ is divisible by $t$. Hence, $s$ is the smallest positive integer such that $ns$ is divisible by $t$, which is $t/\gcd(t,n)$. 
\end{proof}

\section{Future Directions}

It would be nice to have examples of graphs other than forests where toric promotion displays interesting behavior. It could also be interesting to find toric analogues of other appearances of promotion in the literature. For example, one might consider toric analogues of the promotion Markov chains from \cite{Ayyer, Poznanovic}. 

Given an $n$-vertex graph $G=(V,E)$ and a permutation $\pi=\pi_1\cdots\pi_n\in S_n$ written in one-line notation, let us define $\TPro_\pi\colon\Lambda_G\to\Lambda_G$ by $\TPro_\pi=\tau_{\pi_n}\tau_{\pi_{n-1}}\cdots\tau_{\pi_2}\tau_{\pi_1}$. Thus, $\TPro_{12\cdots n}$ is just the toric promotion operator from Definition~\ref{Def:ToricPro}. For $1\leq h\leq \left\lfloor n/2\right\rfloor$, let $\zeta^{(h)}$ denote the permutation $12\cdots (n-h) n(n-1)\cdots (n-h+1)$. In particular, $\zeta^{(1)}$ is the identity permutation $12\cdots n$. One can show that if $\pi\in S_n$, then $\TPro_\pi$ is conjugate in the symmetric group on $\Lambda_G$ to an operator of the form $\TPro_{\zeta^{(h)}}$ for some $1\leq h\leq\left\lfloor n/2\right\rfloor$. Hence, to understand the orbit structures of the operators $\TPro_\pi$, it suffices to understand the orbit structures of the operators $\TPro_{\zeta^{(h)}}$. Along these lines, we have the following conjecture when $G$ is a path; note that the $h=1$ case of this conjecture is a special consequence of Theorem~\ref{Thm:Forest}.

\begin{conjecture}
If $G$ is a path graph with $n$ vertices and $1\leq h\leq \left\lfloor n/2\right\rfloor$, then the operator $\TPro_{\zeta^{(h)}}\colon\Lambda_G\to\Lambda_G$ has order $h(n-h)$. 
\end{conjecture}

More generally, one could consider operators of the form $\tau_{i_m,j_m}\cdots\tau_{i_2,j_2}\tau_{i_1,j_1}\colon\Lambda_G\to\Lambda_G$ for some indices $i_1,\ldots,i_m,j_1,\ldots,j_m$ with $i_\ell\neq j_\ell$ for all $\ell$. This level of generality is probably far too great to expect any noteworthy behavior, but it could be fruitful to explore whether there are specific cases that produce interesting results.

\section{Acknowledgments}
The author thanks Sam Hopkins, Alexander Postnikov, Tom Roby, and Martin Rubey for helpful discussions. The author was supported by a Fannie and John Hertz Foundation Fellowship and an NSF Graduate Research Fellowship (grant number DGE 1656466). The author also thanks the anonymous referees for very helpful comments that improved the presentation of this article.

\end{document}